\documentclass[12pt]{amsart}
\usepackage{amssymb,amsmath,amsthm}
\usepackage{enumerate}
\usepackage[dvips]{graphicx}
\usepackage[usenames]{color}
\newtheorem{theorem}{Theorem}
\newtheorem{lemma}[theorem]{Lemma}

\newtheorem{proposition}[theorem]{Proposition}
\newtheorem{definition}[theorem]{Definition}

\theoremstyle{remark}

\newtheorem*{remark}{Remark}
\newtheorem*{remarks}{Remarks}
\newtheorem*{example}{Example}
\numberwithin{theorem}{section} \numberwithin{equation}{section}

\newcommand{\mfp}{\mathfrak{p}}
\newcommand{\mfq}{\mathfrak{q}}

\newcommand{\mfa}{\mathfrak{a}}

\newcommand{\calP}{\mathcal{P}}

\newcommand{\mfb}{\mathfrak{b}}

\newcommand{\calO}{\mathcal{O}}

\newcommand{\C}{\mathbb{C}}

\newcommand{\calS}{\mathcal{S}}
\newcommand{\Z}{\mathbb{Z}}
\newcommand{\N}{\mathbb{N}}

\newcommand{\textmod}{{\text {\rm mod}}}

\newcommand{\res}{{\text {\rm res}}}

\newcommand{\rcn}{h_{\mfq}}

\newcommand{\sump}{\sideset{}{'}\sum}

\begin{document}
\title[Bubbles of Congruent Primes]
{Bubbles of Congruent Primes}

\author{Frank Thorne}
\address{Department of Mathematics, University of South Carolina,
1523 Greene Street, Columbia, SC 29208}
\email{thorne@math.sc.edu}

\begin{abstract}

In \cite{shiu}, Shiu proved that if $a$ and $q$ are arbitrary coprime
integers, then there exist arbitrarily long strings of consecutive primes
which are all congruent to $a$ modulo $q$. We generalize Shiu's theorem
to imaginary quadratic fields, where we prove the existence
of ``bubbles'' containing arbitrarily many primes which are all, up to 
units, congruent to $a$ modulo $q$.
\end{abstract}

\maketitle

\section{Introduction and Statement of Results}
In 1997, Shiu \cite{shiu} proved that if
$a, q$, and $k$ are arbitrary integers with $(a, q) = 1$,
there exists a string of $k$ consecutive primes
$$p_{n + 1} \equiv p_{n + 2} \equiv \dots \equiv p_{n + k} \equiv a \ (\textmod \ q).$$
(Here $p_n$ denotes the $n$th prime.) Furthermore, if $k$ is sufficiently
large in terms of $q$, these primes can be chosen to
satisfy the bound\footnote{
In Shiu's statement of his results,
the initial $1/\phi(q)$ in (\ref{shiu_bound}) and
the requirement that $k$ be large are omitted, and the
implied constant in (\ref{shiu_bound}) is allowed to depend on $q$. A careful
reading of his proof shows that the dependence on $q$ may be controlled as stated.}
\begin{equation}\label{shiu_bound}
\frac{1}{\phi(q)} \bigg( \frac{\log \log p_{n + 1} \log \log \log \log p_{n + 1}}{(\log
\log \log p_{n + 1})^2} \bigg)^{1/\phi(q)} \ll k,
\end{equation}
uniformly in $q$.

In this paper we prove an analogous statement for imaginary quadratic fields.
If $K$ is such a field,
then the ring of integers $\mathcal{O}_K$ forms a lattice in
$\mathbb{C}$, and the primes of $\mathcal{O}_K$ can be naturally visualized as
lattice points. In this setting one may ask 
whether there are clumps of primes, all of which lie in a fixed arithmetic progression.
We prove that this is indeed the case, up to multiplication by units: 
\begin{theorem}\label{thm_shiu}
Suppose $K$ is an imaginary quadratic field, $k$ is
a positive integer, and $a$ and $q$ are elements of $\calO_K$ with
$q \neq 2$ and $(a, q) = 1$. Then there exists a ``bubble''
\begin{equation}\label{bubble}
B(r, x_0) := \{x \in \C : |x - x_0| < r \}
\end{equation}
with at least $k$ primes, such that all the primes in this bubble
are congruent to $ua$ modulo $q$
for units $u \in \calO_K$.
Furthermore, for $k$ sufficiently large (in terms of $q$ and $K$), $x_0$ can be chosen to satisfy
\begin{equation}\label{eqn_shiu_bound}
\frac{1}{\phi_K(q)} \bigg( \frac{\log \log |x_0| \log \log \log \log |x_0|}
{(\log \log \log |x_0|)^2 }\bigg)^{\omega_K/h_K \phi_K(q)} \ll k.
\end{equation}
The implied constant is absolute.
\end{theorem}
Here $\omega_K$ denotes the number of units in $\calO_K$, $h_K$
is the class number of $K$, and
$\phi_K(q) := |(\calO_K / (q))^\times|.$

\begin{remarks} 
The unit $u$ will not necessarily be the same for each prime in the bubble
(\ref{bubble}). It would be desirable to obtain a version of Theorem
\ref{thm_shiu} where each prime is congruent to $a$ modulo $q$, without the
ambiguity involving units. Unfortunately, this
ambiguity appears to be unavoidable given our 
methods of proof.

The restriction that $q \neq 2$ is not severe; to obtain
prime bubbles modulo 2 we may take (for example) $q =4$.
For the reason behind this restriction, see Lemma \ref{units_lemma}.
\end{remarks}

\begin{example} Let $K = \mathbb{Q}(i)$, $q = 5 + i$, and $a = 1$.
A computer search reveals that the ball of radius $\sqrt{7.5}$ centered
at $2 + 17i$ contains three primes, all of which are congruent to $\pm 1$
or $\pm i$ modulo $q$. Similarly the ball of radius $\sqrt{23.5}$ centered
at $59 + 779i$ contains six primes, all of which are congruent to $\pm 1$
or $\pm i$. Theorem \ref{thm_shiu} establishes the existence of infinitely
many such balls, with $\omega_K / \phi_K(q) = 1/3.$
\end{example}

The proof of Theorem \ref{thm_shiu} is an adaptation of 
Shiu's original proof \cite{shiu}, which we now summarize.\footnote{
We describe a simplified version of Shiu's argument which proves \eqref{shiu_bound} for all $a$;
Shiu proves a better bound than \eqref{shiu_bound} for certain moduli $a \ (\textmod \ q)$.}
Given $a$ and $q$ with $(a, q) = 1$, Shiu constructs a modulus $Q(y)$ such that most
integers in an interval $[1, yz]$ which are coprime to $Q(y)$ are congruent to $a$ modulo $q$.
He then constructs a ``Maier matrix'', the rows of which are short intervals, and the columns of
which are arithmetic progressions modulo $Q(y)$. By an appropriate version of the prime number
theorem for arithmetic progressions \eqref{pnt_gallagher}, most primes in the matrix are congruent to $a \ (\textmod \ q)$.
A counting argument establishes the existence of strings of congruent primes.

In adapting Shiu's proof to imaginary quadratic fields we encounter two difficulties. The first is
the failure of unique factorization. Shiu's argument relies on the unique factorization
of positive integers into positive primes, and we encounter obstructions from both the unit group
(there is no analogue of ``positive'') and the class group. The obstruction from the unit group
seems unavoidable, so we incorporated it into our results. We can handle the
class group, however, and we prove an analogue of \eqref{pnt_gallagher} for principal prime ideals. We introduce 
an {\itshape ad hoc} definition of congruences on ideals; namely, we write $\mfp \equiv a \ (\textmod \ q)$ if
$\mfp$ is principal and any generator is congruent to $a \ (\textmod \ q)$. With this definition, 
we prove that there are sufficiently many prime ideals $\equiv a \ (\textmod \ q)$ to make Shiu's
argument work.

The second difficulty is geometric. Shiu's construction exhibits a string of primes, {\itshape almost} 
all of which are congruent to $a \ (\textmod \ q)$, after which finding a substring of primes $\equiv a \ (\textmod \ q)$ is trivial.
The two-dimensional analogue of this construction is no longer trivial: we find a ``bubble''
in the complex plane containing many ``good'' primes $\equiv ua \ (\textmod \ q)$ and few bad primes, 
and we want a smaller bubble containing only good primes. To obtain this,
we count bad primes in larger bubbles than good primes,
obtaining concentric bubbles in the complex plane. A combinational geometry argument 
(see Section \ref{sec_geometry}) then
allows us to find a bubble containing only good primes.

Generally speaking, the results of this paper indicate that the Maier matrix method ``works''
for imaginary quadratic fields (at least), and we believe that it should be possible to prove
the existence of various irregularities in the distribution of the primes of $\calO_K$,
in analogy with results for $\mathbb{Z}$ obtained by Maier \cite{maier_short_intervals},
Granville and Soundararajan \cite{GS}, and others. (We refer to
 the survey article of Granville \cite{granville} for an interesting
overview of the method and additional related results.) This does present other difficulties however, and
in any case we have not pursued this further here.

The outline of the paper is as follows. In Section \ref{sec_zero_density}
we prove several results related to the distribution of prime ideals in arithmetic progressions.
The most important of these is
a version of the prime number theorem for arithmetic progressions in quadratic
fields (Theorem \ref{pnt_ap_good}), and we closely follow Gallagher \cite{gallagher} for the proof.
In Section
\ref{sec_geometry} we present the combinatorial geometry argument which allows to find bubbles
containing exclusively good primes. We conclude with the proof of Theorem
\ref{thm_shiu} in Section \ref{sec_proof_shiu}.
\\
\\
\bf Setup and notation\rm. We assume $K$ is an imaginary quadratic
field with a fixed embedding $K \rightarrow \mathbb{C}$,
with class number $h_K$ and $\# \calO_K^{\times}
= \omega = \omega_K \in \{2, 4, 6\}$.
Any $K$-dependence of implicit constants occuring in our results
will be explicitly noted.

We will write $\mfq = (q)$ throughout, and where it does not lead to
ambiguity we will refer to $\mfq$ and $q$ interchangeably.
We assume that the units of $\calO_K$ all represent distinct residue classes mod $\mfq$; 
by Lemma \ref{units_lemma}, this only excludes
three choices for $\mfq$. We further assume that the units do not
represent all reduced residue classes modulo $\mfq$; 
if this happens then Theorem \ref{thm_shiu} is trivial.

As $K$ will be fixed, we will simply write $\phi(q)$ (or $\phi(\mfq)$) 
for $\phi_K(q) := |(\calO_K / (q))^*|.$ We will also write $\rcn$
for $h_K \phi(q)/\omega$, the size of the ray class group.

Our methods oblige us to define congruences on ideals.
For an ideal $\mathfrak{b}$ of $\calO_K$ and $a, q \in \calO_K$, we say that
$\mathfrak{b} \equiv a \ (\textmod \ q)$ if $\mathfrak{b}$ is principal and
$b \equiv a \ (\textmod \ q)$ for any $b$ for which $\mathfrak{b} = (b)$.
If $\mathfrak{b} \equiv a$,
then $\mathfrak{b} \equiv ua$ for any unit $u \in \calO_K^{\times}$. Equivalently, we see
that $\mathfrak{b} \equiv a \ (\textmod \ q)$ if 
$\mathfrak{b}$ and $(a)$ represent the same class in the ray
class group $H^{(q)}$. For nonprincipal $\mathfrak{b}$ we say that
$\mathfrak{b} \not \equiv a \ (\textmod \ q)$ for any $a$.

\section*{Acknowledgements}
I thank Bob Hough, Jorge Jim\'enez-Urroz, and an anonymous referee for useful advice and suggestions.
In particular, I thank Hough for suggesting an improvement to a previous version of 
Proposition \ref{good_bad_points}.

This work was part of my graduate thesis; I thank my advisor Ken Ono for his many useful suggestions, as well as the NSF
for financial support.

\section{Prime ideals in arithmetic progressions}\label{sec_zero_density}

One standard ingredient in the Maier matrix method is a theorem of Gallagher
(\cite{gallagher}; see also \cite{maier_large_gaps}, Lemma 2), who proved
that
\begin{equation}\label{pnt_gallagher}
\pi(x; q, a) = (1 + o_{D}(1)) \frac{x}{\phi(q) \log x},
\end{equation}
uniformly in $x \gg q^D$, for a suitably large (infinite) set of moduli $q$. This result
serves as a substitute for the Riemann hypothesis, and allows one to count
the number of primes in Maier matrices in different arithmetic progressions.

The main goal of this section is generalize this result to 
imaginary quadratic fields. We will work with prime ideals rather than prime elements, to
preserve unique factorization, and it will be necessary (if a bit unnatural) to describe
the distribution of prime ideals in congruence classes.

\begin{definition}
We write $\pi_1(x; q, a)$ for the number of principal prime ideals $\mathfrak{p}$
of norm $\leq x$, such that $p \equiv a \ (\textmod \ q)$ for some generator
$p$ of $\mathfrak{p}$.
\end{definition}

We will estimate $\pi_1(x; q, a)$ using analytic techniques applied to Hecke $L$-functions.
We first recall the necessary definitions and terminology.

In place of $\calO_K/\mfq$ we begin with the
\itshape ray class group \upshape modulo $\mfq$
\begin{equation}\label{def_hq}
H^{\mfq} := J^{\mfq}/P^{\mfq},
\end{equation}
where $J^{\mfq}$ is the group of all fractional ideals coprime to $\mfq$,
and $P^{\mfq}$ is the group of principal fractional ideals $(a) = (b)(c)^{-1}$
with $b, c \in \calO_K$ and $b \equiv c \equiv 1 \mod \mfq$.
If we write $J^{\mfq}_1$ for the group of principal fractional
ideals coprime to $\mfq$, then
$J^{\mfq}_1 / P^{\mfq}$ is in one-to-one correspondence with the set of sets of
reduced residue classes modulo $\mfq$
\begin{equation}\label{count_residue_classes}
\{ua: (a, \mfq) = 1, u \in \calO_K^{\times} \},
\end{equation}
where $a$ is a fixed in each set and $u$ ranges over all units of $\calO_K$.
The proof of Theorem \ref{thm_shiu} will exhibit bubbles of
prime elements $p$, such that the ideals $(p)$ all lie in a
fixed class in $J^{\mfq}_1 / P^{\mfq}$.

Suppose henceforth that $\mfq \not \in \big\{ (2), \big(\frac{-3 \pm \sqrt{-3}}{2}\big) \big\}$ and $\phi(\mfq) > 1$.
Then the size of the ray class group is given by the following simple formula.
\begin{lemma}\label{units_lemma}
If  $\mfq \not \in \big\{ (2), \big(\frac{-3 \pm \sqrt{-3}}{2}\big) \big\}$ and $\phi(\mfq) > 1$, then we have
\begin{equation}\label{eqn_units}
\rcn := |H^{\mfq}| = h_K \phi(\mfq)/\omega,
\end{equation}
 where $h_K$ is the class number of $K$, and $\omega \in \{2, 4, 6\}$
denotes the number of units of $\calO_K$.
\end{lemma}

This is not difficult to show: as $J^{\mfq} / J^{\mfq}_1$ is isomorphic to the usual
class group, \eqref{eqn_units} follows by showing that there are $\phi(\mfq)/\omega$
sets counted in \eqref{count_residue_classes}, which in turn follows by showing that
$u - 1 \not \in \mfq$ for each unit $u \neq 1$ of $\calO_K$. This latter fact is easily
checked (given the conditions on $\mfq$) and we omit the details.

\begin{remark}
In the case where $\mfq = \big(\frac{-3 \pm \sqrt{-3}}{2}\big)$
and $K = \mathbb{Q}(\sqrt{-3})$, the units of $\calO_K$ 
cover all reduced residue classes mod $\mfq$ and so the statement of Theorem \ref{thm_shiu} is empty.
\end{remark}

From the group $H^{\mfq}$ we obtain \itshape Hecke characters \upshape $\chi$ of $K$ by lifting
any character $\chi$ of $H^{\mfq}$ to $J^{\mfq}$ in the obvious way,
and setting $\chi(\mfa) = 0$ for any $a$ not coprime to $q$.
Throughout, we will only consider Hecke characters obtained
in this fashion. (See, however, Chapter VII.6 of \cite{neukirch} (for example) for a
more general discussion.) The associated \itshape Hecke $L$-functions \upshape
are defined by the equation
\begin{equation}
L(s, \chi) := \sum_{\mfa} \chi(\mfa) (\N \mfa)^{-s},
\end{equation}
where $\mfa$ runs over all integral ideals of $\calO_K$.

Our estimates for $\pi_1(x; q, a)$ will depend on a
zero-free region for the Hecke $L$-functions modulo $\mfq$. For convenience, we formulate 
this hypothesis as {\itshape Hypothesis ZF(C)}:
\begin{definition} 
If $C > 0$, we say that $\mfq$ satisfies Hypothesis ZF(C)
if none of
the Hecke $L$-functions modulo $\mfq$ have a zero in the region
\begin{equation}\label{nice_zero_free}
\sigma > 1 - C / \log[(\N q)( |t| + 1)].
\end{equation}
We say that $q \in \calO_K$ satisfies Hypothesis ZF(C) if the ideal $(q)$ does.
\end{definition}

We will prove the following:

\begin{theorem}\label{pnt_ap_good}
Suppose that $q \in \calO_K$ is
not $u$, $2u$, or $\frac{-3 \pm \sqrt{-3}}{2} u$ for any unit
$u$ of $\calO_K$, and that $q$ satisfies Hypothesis ZF(C) for some $C$.

Then for $D \geq 0$ we have
$$\pi_1(2x; q, a) - \pi_1(x; q, a) = (\omega_K + o_{x, D}(1)) \frac{x}{h_K \phi_K(q) \log x},$$
uniformly in $q$ for $(a, q) = 1$, $\N q \geq |\Delta_K|$, and $x \geq \N q^D$.
\end{theorem}

Here $o_{x, D}(1)$ denotes an error term bounded above by any $\epsilon > 0$,
provided both $x$ and $D$ are chosen sufficiently large. The error term also depends
on $C$, but in the application $C$ will be an absolute constant.

We further remark that
the condition $\N q \geq |\Delta_K|$ is required only if the
$o_{x, D}(1)$ term is to be independent of $K$. Also, the restriction
on $q$ is not serious, as we may find primes in arithmetic progressions $(\textmod \ q')$
for an appropriate multiple $q'$ of $q$.

To use Theorem \ref{pnt_ap_good}, we must prove 
that the zero-free region (\ref{nice_zero_free}) holds for a suitably
large (infinite) set of moduli. To define these moduli 
we introduce the notation
\begin{equation}
\calP(y, q, \mfp_0) := q \prod_{\N \mfp \leq y; \mfp \neq \mfp_0} \mfp.
\end{equation}
\begin{proposition}\label{prop_zero_free}
For all sufficiently large $x$ there exist an integer $y$ and a prime $\mfp_0$ with
$x < \N \calP(y, q, \mfp_0) \ll x \log^3 x$
and $\N \mfp_0 \gg \log y,$ 
such that $q$ satisfies Hypothesis ZF($C_2$) for an absolute constant $C_2$.
\end{proposition}
The proposition and its proof, given in Section \ref{sec_zero_free}, are the direct analogues of Theorem 1 of \cite{shiu}.
Note that the prime $\mfp_0$ is removed to ensure that the Siegel zero doesn't exist.
The definition of  ``sufficiently large'' depends on $K$. We could easily control the
$K$-dependence here, but it would be more difficult in Lemma 
\ref{lemma_1_mod_m} and so we don't bother.

\subsection{Proof of Theorem \ref{pnt_ap_good}}
Theorem \ref{pnt_ap_good} will follow from the following estimate:
\begin{proposition}\label{prop_char_sum}
If $\mfq$ satisfies Hypothesis ZF($C_1$) and $\max(\exp(\log^{1/2} x), \Delta_K) \leq \N \mfq \leq x^b$ for a fixed constant
$b > 0$, then we have
\begin{equation}\label{eqn_char_sum}
\sum_{\chi} \Big| \sum_{\N \mfp \in [x, 2x]} \chi(\mfp) \log(\N \mfp) \Big|
\ll x \exp \Big( -a \frac{\log x}{\log \N \mfq} \Big),
\end{equation}
where the constant $a$ depends only on $C_1$, 
the first sum is over all nonprincipal characters modulo $\mfq$,
and the implied constant is absolute.
\end{proposition}

With additional care, we expect to be able to prove a similar result
for an arbitrary number field $K$.

Theorem \ref{pnt_ap_good} follows from Proposition \ref{prop_char_sum} as follows:
By the orthogonality relations, we have
$$\sum_{\substack{ \N \mfp \in [x, 2x] \\ \mfp \equiv a \ (\textmod \ q)}} \log(\N \mfp)
= \frac{1}{\rcn} \sum_{\N \mfp \in [x, 2x]} 
\sum_{\chi \ (\textmod \ \mfq)} \bar{\chi}(a) \chi(\mfp) \log(\N \mfp)$$
$$= \frac{1}{\rcn} \sum_{\N \mfp \in [x, 2x]} \log(\N \mfp) +
O\bigg( \frac{1}{\rcn} \sum_{\chi \neq \chi_0} \bigg| \sum_{\N \mfp \in [x, 2x]}
\chi(\mfp) \log(\N \mfp) \bigg| \bigg),$$
and for $x \leq \exp((\log \N \mfq)^2)$, the result now follows from the prime ideal theorem and Proposition \ref{prop_char_sum}.

For the (easier) range $x > \exp((\log \N \mfq)^2)$, a proof can be given as follows.
Take $T = \exp((\log x)^{3/4})$ in the proof of Proposition
\ref{prop_char_sum}, and the quantity in \eqref{eqn_char_sum} is
$\ll x \exp(-a (\log x)^{1/4})$, which suffices for our result.

It therefore suffices to prove Proposition \ref{prop_char_sum}, and we will closely follow Gallagher
\cite{gallagher}. Gallagher proves a similar result for
Dirichlet $L$-functions, but with an additional sum over moduli $q$. He deduces
his result from a log-free zero-density estimate for these $L$-functions,
and in our case the appropriate zero-density estimate has been proved\footnote{
This is stated, in a slightly different form, after the main theorem of \cite{fogels}. Note that Fogels
published a corrigendum to \cite{fogels}, but that it does not affect the statement of the main results.}
by Fogels \cite{fogels}:
\begin{proposition}[Fogels]\label{prop_fogels}
We have for any $\mfq \in \calO_K$ and any $T \geq \Delta_K \N \mfq$
\begin{equation}\label{eqn_fogels}
\sum_{\chi} N_{\chi} (\alpha, T) \leq T^{c(1 - \alpha)}.
\end{equation}
Here $N_{\chi} (\alpha, T)$ denotes the number of zeroes $\rho = \beta + it$ 
of $L(s, \chi)$ with $\alpha < \beta < 1$ and $|t| < T$,
$\chi$ ranges over all characters modulo $\mfq$, 
$\Delta_K$ is the discriminant of $K$, and
$c$ is (for quadratic fields) an absolute constant.
\end{proposition}

\begin{proof}[Proof of Proposition \ref{prop_char_sum}]
At the outset, we choose
$T = (\N \mfq)^2 \leq x^{1/2c}$, which is an acceptable choice in all
of our estimates.

By standard analytic techniques (see (5.53) and (5.65) of \cite{IK}), we have
\begin{equation}\label{ik_sum}
\sum_{\N \mfa \in [x, 2x]} \chi(\mfa) \Lambda(\mfa)
= \delta_\chi x - \sum_\rho \frac{(2x)^{\rho} - x^{\rho}}{\rho} + O\Big(\frac{x \log^2  x}{T}\Big),
\end{equation}
where $\delta_{\chi}$ is 1 or 0 according to whether $\chi$
is principal or not, $\Lambda(\mfa) := \log(\N \mfp)$ if $\mfa$ is a power of some prime $\mfp$
and 0 otherwise, and $\rho$ ranges over all the zeroes $\rho = \beta + it$ of $L(s, \chi)$ in the 
critical strip with $|t| < T$.

We observe that for each $\rho = \beta + it$,
$$\frac{(2x)^{\rho} - x^{\rho}}{\rho} \ll x^{\beta}.$$
The terms where $\mfa$ is a prime
power (but not a prime) contribute $\ll x^{1/2}$ to the sum (\ref{ik_sum})
and so may be absorbed into the error
term for $T \leq x^{1/2}$. Therefore, for nonprincipal $\chi$ we see that
$$\sum_{\N \mfp \in [x, 2x]} \chi(\mfp) \log(\N \mfp) \ll \sum_{\rho} x^{\beta}
+ \frac{x \log^2 x}{T}.$$
Therefore,
$$\sum_{\chi \neq \chi_0} \Big| \sum_{\N \mfp \in [x, 2x]} \chi(\mfp) \log(\N \mfp) \Big|
\ll \sum_{\chi \neq \chi_0} \sum_{\rho} x^{\beta}
+ \frac{x \log^2 x (\N \mfq)}{T}.$$
The sum over $\chi$ and $\rho$ on the right is
\begin{equation}\label{big_integral}
-\int_0^1 x^{\sigma} d_{\sigma}\Big(\sum_{\chi \neq \chi_0} N_{\chi}(\sigma, T)\Big)
= -x^{\sigma} \Big( \sum_{\chi \neq \chi_0} N_{\chi}(\sigma, T) \Big) \bigg|_0^1
+ \int_0^1 x^{\sigma} \log x \Big(\sum_{\chi \neq \chi_0} N_{\chi} (\sigma, T)\Big) d\sigma.
\end{equation}
The first term of (\ref{big_integral}) is (\cite{IK}, Theorem 5.8)
$$\sum_{\chi \neq \chi_0} N_{\chi}(0, T) \ll T \N \mfq\log(T \N \mfq).$$
Using the zero-free region (\ref{nice_zero_free}) and Proposition
\ref{prop_fogels},
we see that the second term of (\ref{big_integral}) is
$$\ll \int_0^{1 - C_1/\log[(\N \mfq)(T + 1)]} (x^{\sigma} \log x) T^{c (1 - \sigma)}
d \sigma.$$
Evaluating the integral above and recalling that $T \leq x^{1/2c}$, this second
term is
$$\ll x \exp\biggl( -\frac{C_1}{2} \frac{\log x}{\log[(\N \mfq)(T + 1)]}\biggr).$$
We conclude from all these estimates that
\begin{multline*}
\sump_{\chi} \Big| \sum_{\N \mfp \in [x, 2x]} \chi(\mfp) \log(\N \mfp) \Big|
\ll \\ \frac{(x \log^2 x) \N \mfq}{T} +
T \N \mfq \log (T \N \mfq) +
 x \exp\biggl( -\frac{C_1}{2} \frac{\log x}{\log[(\N \mfq)(T + 1)]}\biggr).
\end{multline*}
With the choice $T = (\N \mfq)^2$ and the hypothesis that $\max(\exp(\log^{1/2} x), |\Delta_K|) 
\leq \N \mfq
\leq \min(x^{1/4c}, x^{1/4})$, we obtain the proposition.
\end{proof}
\subsection{Proof of Proposition \ref{prop_zero_free}}\label{sec_zero_free}

The proof follows Theorem 1 of \cite{shiu}.
We require the following zero-free region for Hecke $L$-functions,
also due to Fogels \cite{fogels_zerofree}:
\begin{lemma}[Fogels]\label{prop_fogels_zero_free}
Assume that $\mfa$ is an ideal of $\calO_K$ with $|\Delta_K \N \mfa|$ sufficiently
large. Then $\mfa$ satisfies Hypothesis ZF$(C_3)$ for an absolute constant $C_3$, 
with the possible exception
of a single zero $\beta$ of one Hecke $L$-function $L(s, \chi)$ modulo $\mfa$.
If $\beta$ exists then it must be real and satisfy
\begin{equation}\label{fogels_siegel_zero}
\beta < 1 - (|\Delta_K| \N \mfa)^{-4}.
\end{equation}
\end{lemma}
\begin{remark}
The above results in fact hold for an arbitrary number field $K$. In this
case $C$ depends on the degree of $K$, and the exponent $-4$ in (\ref{fogels_siegel_zero})
should be replaced with $-2 [K : \mathbb{Q}]$. As elsewhere in this paper,
``sufficiently large'' is allowed to depend on $K$ (even for quadratic fields).
\end{remark}
\begin{proof}[Proof of Proposition \ref{prop_zero_free}]

Consider the product
\begin{equation}
\calP'(y, q) := q \prod_{\N \mfp \leq y} \mfp,
\end{equation}
and suppose that an exceptional character mod $\calP'(y, q)$ exists; i.e.,
suppose that there exists a character $\chi_1$ mod $\calP'(y, q)$
whose $L$-function has a real zero $\beta$ in the range
\begin{equation}\label{beta_range}
1 \geq \beta \geq 1 - \frac{C_3}{\log(|\Delta_K| \N \calP'(y, q))}.
\end{equation}
Write $\chi'_1$ (mod $\calP''$) for the primitive character inducing $\chi_1$,
so that $\calP'' | \calP'(y, q).$ Then comparing (\ref{beta_range})
with (\ref{fogels_siegel_zero}) we see\footnote{
If $|\Delta_K|$ is small
it might be the case that $\calP''$ is of too small norm to apply
\eqref{fogels_siegel_zero}. For each such $K$ we may choose a fixed ideal 
$\mathfrak{b}$ of sufficiently large norm, and write $\chi_1''$ for the
character modulo $\mathfrak{b} \calP''$ induced by $\chi'_1$. The associated
$L$-function will have a zero at the same spot, and we conclude that
$\N(\mfb \calP'') \gg \frac{1}{|\Delta_K|} (\log \N \calP'(y, q))^{1/4}$.
As $\mfb$ is fixed for each $K$, this implies that
$\N \calP'' \gg \frac{1}{|\Delta_K|} (\log \N \calP'(y, q))^{1/4}$
as well.}
that $\N \calP'' \gg \frac{1}{|\Delta_K|}(\log \N \calP'(y, q))^{1/4}.$
We thus see that for sufficiently large $y$ (in terms of $q$),
$\calP''$ will have a prime divisor $\mfp_0$ satisfying
$\mfp_0 \gg \log (\N \calP'') \gg \log \log (\N \calP'(y, q)) \gg \log y.$

We claim that there can be no character $\chi_2$ modulo $\calP(y, q, \mfp_0)$
whose $L$-function has a real zero in the region
\begin{equation}\label{beta_range_2}
\beta' > 1 - \frac{C_3}{2 \log(|\Delta_K| \N \calP(y, q, \mfp_0))}.
\end{equation}
Assuming this for now, we see that $\calP(y, q, \mfp_0)$ satisfies 
Hypothesis $ZF(C_2)$ with
with $C_2 := C_3/4$, provided that $y$ is large
enough so that $\N \calP(y, q, \mfp_0) \geq |\Delta_K|$.
To prove our claim, suppose such a $\chi_2$ exists.
Then $\beta'$ will be in the region (\ref{beta_range}), and as $\chi_2$ and
$\chi'_1$ induce different characters modulo $\calP'(y, q)$, $\beta$
and $\beta'$ will be zeroes to distinct $L$-functions modulo $\calP'(y, q)$
in the region \eqref{beta_range}, contradicting Lemma
\ref{prop_fogels_zero_free}.

If no exceptional character mod $\calP'(y, q)$ exists,
we choose $\mfp_0$ to be any
prime divisor of $\calP'(y, q)$ of norm $\geq \log y$. 
We again take $C_2 := C_3/4$ and see that (for large $y$)
no $L$-function modulo $\calP(y, q, \mfp_0)$
will have a zero in the region (\ref{beta_range_2}). 

To conclude, we must show that we can find a $\calP(y, q, \mfp_0)$ in each
range $x < \N \calP(y, q, \mfp_0) \ll x \log^3 x$. In quadratic fields
there can exist at most two distinct primes of the same norm. For a fixed
large $y$, let $y' > y$ be minimal so that $\calP(y', q) \neq \calP(y, q)$.
Then $\N \calP(y', q)/\N \calP(y, q)
\leq (y')^2 = (1 + o(1)) \log^2(\N \calP(y', q))$,
so for any large $x$ we can find $y$ with
$2 x \log x < \N \calP(y, q) < 3 x \log^3 x.$ Removing a prime $\mfp_0$
from our product we see that necessarily $\N \mfp_0 \leq y
= (1 + o(1)) \log x$ and so
$x < \N \calP(y, q, \mfp_0) \ll x \log^3 x$, as desired.
\end{proof}
\subsection{Additional lemmas}\label{sec_addl_lemmas}
We need two additional lemmas on the distribution  of ideals with certain restrictions on their prime factors.

\begin{lemma}\label{lemma_1_mod_m}
Let $\mathcal{S}(x)$ denote the number of ideals of 
norm $\leq x$ whose prime (ideal) factors are all
$\equiv 1 \ (\textmod \ \mfq)$. Then
\begin{equation}\label{eqn_1_mod_m}
\mathcal{S}(x) = (C_{\mfq} + o_{\mfq}(1)) x (\log x)^{-1 + 1/\rcn},
\end{equation}
where
\begin{equation}
C_{\mfq} := \frac{1}{\Gamma(1/\rcn)} \lim_{s \rightarrow 1^+}
\biggl[ (s - 1)^{1/\rcn} \prod_{\mfp \equiv 1 \ (\textmod \ \mfq)} 
\Big(1 - \frac{1}{(\N \mfp)^{-s}}\Big)^{-1} \biggr].
\end{equation}
\end{lemma}
\begin{proof}
This is a generalization of Landau's work on sums of two squares, and
also of Lemma 3 of \cite{shiu}. Write
\begin{equation}
F(s) := \prod_{\mfp \equiv 1 \ (\textmod \ \mfq)} \Big(1 - \frac{1}{(\N \mfp)^{-s}}\Big)^{-1}.
\end{equation}
Then by a Tauberian theorem due to Raikov (\cite{CM}, Theorem 2.4.1),
the asymptotic (\ref{eqn_1_mod_m}) follows if we can write
$$F(s) = \frac{H(s)}{(s - 1)^{1/\rcn}}$$
for a function $H(s)$ which is holomorphic and nonzero in the region
$\Re(s) \geq 1$, with
$$C_{\mfq} = \frac{H(1)}{\Gamma(1/\rcn)}.$$
We write
\begin{equation}
\Theta(s) := \frac{\prod_{\chi \ (\textmod \ \mfq)} L(s, \chi)}{F(s)^{\rcn}},
\end{equation}
and computing the Dirichlet series expansion for $\log \Theta(s)$ 
(exactly as in \cite{shiu}) 
we conclude that $\Theta(s)$ is holomorphic for $\Re(s) > \frac{1}{2}$.
The product $\prod_{\chi  \ (\textmod  \ \mfq)} L(s, \chi)$ has a simple pole at $s = 1$, and is otherwise
holomorphic and nonzero in $\Re(s) \geq 1$. The result follows.
\end{proof}
We now need a result from the theory of `smooth' numbers, i.e., numbers
whose prime factors are all sufficiently small. (See, for example,
Chapter III.5 of Tenenbaum's book \cite{tenenbaum} for a general introduction
to the theory.)  Here we require a result
for `smooth' algebraic integers in $K$.

\begin{lemma}\label{lemma_Psi}
Let $\Psi_K(x, y)$ be the number of ideals of norm $< x$ which are composed only
of primes with norm $< y$, and write $u := \log x / \log y$. Then for
$1 \leq u \leq \exp(c (\log y)^{3/5 - \epsilon})$ (for a certain constant $c$) we have 
\begin{equation}\label{eqn_Psi}
\Psi_K(x, y) \ll_K x \log^2 y \exp(-u(\log u + \log \log u + O(1))).
\end{equation}
\end{lemma}
\begin{proof}
This follows immediately by comparing results of de Bruijn \cite{deB} and
Krause \cite{krause}.
de Bruijn proved (\ref{eqn_Psi}) for $K = \mathbb{Q}$.
For general $K$, Krause proved an asymptotic formula for $\Psi_K(x, y)$ in terms
of the Dickman function, and Krause's result implies in particular that for $u$
in the range specified,
$$\lim_{x, y \rightarrow \infty} \frac{\Psi_K(x, y)}{\Psi(x, y)} = \res_{s = 1}
\zeta_K(s),$$ where $\zeta_K(s)$ denotes the Dedekind zeta function.
The lemma then follows immediately.
\end{proof}

\section{Bubbles of Good and Bad Points}\label{sec_geometry}

Suppose we are given a set of integers, $g$ of which are ``good'' and $b$ of which
are ``bad''. Trivially, this set contains a string of $\gg g/b$ consecutive good integers.
In this section we prove a two-dimensional analogue of this statement.

We formlulate our result as a general proposition in combinatorial geometry.
Suppose some circle in the plane contains (in its interior) $g$ ``good'' points and $b$ ``bad'' points. 
(In our application, these will be 
prime elements of $\calO_K$ congruent and not congruent to $ua \ (\textmod \ q)$, respectively.)
We
would like to find a smaller circle containing $\gg g/b$ good points and no bad ones. We must find
this entirely within the original circle, as there may be additional bad points outside this circle
which we have not counted.

This is too much to ask for in general; for example, we cannot find such a smaller circle 
if the good points are all close to the boundary and the bad points are spread evenly.
To avoid such counterexamples, we count good and bad points in concentric circles
as follows:

\begin{proposition}\label{good_bad_points}Suppose the plane contains some number of ``good'' and ``bad'' points, that the
unit circle contains $g$ good points, and that the circle $|z| < 3$ contains $b$ bad points. 
Then there exists some circle in the plane containing 
$> g/(2b + 12)$ good points and no bad points.
\end{proposition}

\begin{remark} In our application to the proof of Theorem \ref{thm_shiu},
$b$ and $g$ will be large with $b = o(g)$.
The construction will be scaled and translated to appropriate regions of
the complex plane.
\end{remark}

For the proof we require the following geometric construction:
\begin{lemma}\label{lem_delaunay}
Let $\mathcal{P}$ be a set of points $N$ in the plane, not all collinear. Then 
there exists a triangulation (called a {\itshape Delaunay triangulation})
of $\mathcal{P}$, such that no point of $\mathcal{P}$ is inside the
circumcircle of any triangle. This triangulation consists of $2N - 2 - k$ triangles, where
$k$ is the number of points in $\mathcal{P}$ lying on the boundary of the convex
hull of $\mathcal{P}$.
\end{lemma}
See, e.g., Chapter 9 of \cite{comp_geom} for a proof of this. Observe also that if
all points of $\mathcal{P}$ are collinear, then Proposition \ref{good_bad_points}
is trivial.

\begin{proof}[Proof of Proposition \ref{good_bad_points}]

The proof is by geometric construction. Write $V$ for the set
of all bad points of distance
less than 3 from the origin, together with the 7-gon consisting of the points 
points $2 e^{2 \pi i n/7}$, for $n \in \Z$.

Construct the Delaunay triangulation $T$ of $V$, let $\mathcal{C}$ be the set of
circumcircles of all triangles in $T$, and let $\mathcal{C}' \subseteq \mathcal{C}$ 
be those circles which nontrivially intersect the interior of the unit circle. By construction,
no circle in $\mathcal{C'}$ contains any point of $V$, and the circles in $\mathcal{C'}$
cover the interior of $C(1)$, with the exception of any bad points.

We claim that every circle in $\mathcal{C'}$ is contained inside $\{ |z| = 3 \}.$
Supposing this for now,
we know that the circles in $\mathcal{C'}$ do not contain any bad points, including
any which lie on or outside $\{ |z| = 3 \}$. These circles do contain all
of the good points in the unit circle, and it follows that one such circle
contains $\geq g/|\mathcal{C}'|$ good
points. Lemma \ref{lem_delaunay} implies that
$|\mathcal{C}'| < 2(b + 7) - 2$ as required.

To prove that every circle in $\mathcal{C'}$ is contained inside $\{ |z| = 3 \},$
suppose that $C$ is a counterexample. Then
$C$ or its interior will contain points $P_1$ and $P_3$ with $|P_1| = 1$ and $|P_3| = 3$.
Furthermore, we 
may take these points to be on the ray from the origin going through the center of $C$.
We also
easily check that $C$ must contain the circle having $\overline{P_1 P_3}$ as its diameter.

For some point $Q$ of the 7-gon, the angle between $\overrightarrow{O Q}$
and $\overrightarrow{O P_3}$ is at most $\pi/7$ and in particular is less than $\pi/6$.
We check that the distance between $Q$ and the midpoint of $\overline{P_1 P_3}$ is then
less than 1, which implies that $Q$ is contained in the interior of $C$, our contradiction.

\end{proof}

\begin{remark}
We thank Bob Hough, who suggested an improvement which improved the statement of
Proposition \ref{good_bad_points} and simplified its proof.
\end{remark}

\section{Proof of Theorem \ref{thm_shiu}}\label{sec_proof_shiu}
We fix $K$, $\mfq = (q)$, and $a$;
we assume that the units of $\calO_K$
do not represent all the reduced residue classes modulo $\mfq$, and that
the residue classes represented are all distinct. Except where noted,
implied constants in our analysis do not depend on $\mfq$. 
We assume a sufficiently large (in terms of $\mfq$ and $K$) integer $x$ is given,
and choose $D > 3$ such that the term $o_{x, D}(1)$ of Theorem \ref{pnt_ap_good}
is bounded by $\frac{1}{2}$.

Our proof consists of three steps. In the first step, we find a modulus $Q$ such that
any $b \in \calO_K$ of small norm which is coprime to $Q$ is very likely to be
congruent to $ua \ (\textmod \ Q)$. In the second step, we use this $Q$ to construct a
Maier matrix of elements of $\calO_K$, such that nearly all of the primes in the matrix
are $\equiv ua \ (\textmod \ Q)$. In the final step, we argue that this Maier matrix contains
a bubble of congruent primes, and bound its size from below.
\\
\\
{\bf The modulus $Q$.} We use Proposition \ref{prop_zero_free} to choose
$y$ and $\mfp_0$ such that
$$x^{1/D} < \N \calP(y, q, \mfp_0) \ll x^{1/D} \log x$$
and such that $\calP(y, q, \mfp_0)$ satisfies Hypothesis $ZF(C_2)$.
We introduce variables $z < y$ and $t < (yz)^{1/2}$ with $z = o(t)$,
and define a set of primes $\mathcal{P}$ as follows:
If $a$ is not congruent to a unit modulo $\mfq$, we define
\begin{equation}
\mathcal{P} := \left\{
\begin{array} {l l} 
\{\mfp: \N \mfp \leq y, \mfp \neq \mfp_0, \mfp \not \equiv 1, a \mod \mfq) \} \\
\ \cup \ \{\mfp: t \leq \N \mfp \leq y, \mfp \neq \mfp_0, \mfp \equiv 1 \mod \mfq \} \\
\ \cup \ \{\mfp: \N \mfp \leq yz/t, \mfp \neq \mfp_0, \mfp \equiv a \mod \mfq \}.
\end{array} \right.
\end{equation}
If $a$ is congruent to a unit modulo $\mfq$, we define instead
\begin{equation}\label{def_P_2}
\mathcal{P} := \left\{
\begin{array} {l l} 
\{\mfp: \N \mfp \leq y, \mfp \neq \mfp_0, \mfp \not \equiv 1 \mod \mfq \} \\
\ \cup \ \{\mfp: t \leq \N \mfp \leq yz/t, \mfp \neq \mfp_0, \mfp \equiv 1 \mod \mfq \}.
\end{array} \right.
\end{equation}
We recall our convention that a nonprincipal prime ideal $\mfp$ is not $\equiv a \ (\textmod \ \mfq)$
for any $\mfq$.

The latter definition (\ref{def_P_2}) is motivated by simplicity, as it allows us
to treat both cases simultaneously. Following Shiu \cite{shiu}, it should
be possible to define $\mathcal{P}$ differently in this case, and modestly
improve our result for a certain subset of moduli $a$.

We further define
\begin{equation}
\mathfrak{Q} = (Q) := \mfq \prod_{\substack{\mfp \in \mathcal{P} \\ \mfp \neq \mfp_1}} \mfp,
\end{equation}
where $\mfp_1$ is any prime ideal with $\log y < \N \mfp_1 \leq y$ for which
$\mathfrak{Q}$ is principal. We may then write $Q$ for any generator
of $\mathfrak{Q}$.

We see that $\mathfrak{Q} | \calP(y, q, \mfp_0)$ and 
$\log(\N Q) \geq \frac{1}{3} \log(\N \calP(y, q, \mfp_0))$.  
Proposition \ref{prop_zero_free} thus implies that the Hecke $L$-functions
modulo $\mathfrak{Q}$ have no zeroes in the
region
\begin{equation}
1 \geq \Re s > 1 - \frac{C_2}{3 \log[(\N Q)(|t| + 1)]},
\end{equation}
as any such zeroes would induce zeroes of
$L$-functions modulo $\calP(y, q, \mfp_0)$ at the same point,
contrary to Hypothesis $ZF(C_2)$ for $\mathcal{P}(y, q, \mfp_0)$.
In other words $Q$ satisfies Hypothesis $ZF(C)$ with 
with $C := C_2/3$,
so that the primes are well-distributed (i.e., Theorem \ref{pnt_ap_good} holds)
in arithmetic progressions modulo $Q$.
\\
\\
{\bf Construction of the Maier matrix.} 
Our construction adapts that of Shiu. In our
case, the geometrical argument given in Section \ref{sec_geometry}
requires us to keep track of more ``bad" primes than ``good". Thus
we define ``bubbles" $B$ and $B'$ consisting of those elements of $\calO_K$
whose norm is less than $yz$ and $9yz$, respectively. We further define
Maier matrices $M$ and $M'$, with $(i, b)$ entry equal to 
the algebraic integer $i Q + b$, where $i$ ranges over all elements of $\calO_K$
with norm in $(\N Q^{D - 1}, 2\N Q^{D - 1})$, and $b$ ranges over
elements of $B$ and $B'$ respectively. We regard $M$ naturally as a submatrix
of $M'$.

We define sets
\begin{equation}
S := \{b \in B; (b, Q) = 1; b \equiv ua \mod q {\text{ for some }} u \in \calO_K^{\times}\}
\end{equation}
and
\begin{equation}
T := \{b \in B'; (b, Q) = 1; b \not \equiv ua \mod q {\text{ for any }} u \in \calO_K^{\times}\}.
\end{equation}
We will prove that $S$ is much larger than $T$.

To estimate $S$, we observe that most elements of $S$ are uniquely determined as elements
of the form $p n$,
where $p$ is a prime of norm $> yz / t$ and is congruent to $ua$ for some unit $u$,
and $n$ is a product of primes congruent to 1 modulo $q$. (There will also be
multiples of $\mfp_0$ and $\mfp_1$, which we ignore.) 
Subdividing dyadically, we see that
$$
|S| \geq \sum_{i = 0}^{\lfloor \frac{\log t}{\log 2} \rfloor - 2}
\Big(\pi_1(2^{i + 1} yz/t; q, ua) - \pi_1(2^i yz/t; q, ua)\Big)
\calS(t/2^{i + 1})
$$
$$
\gg \frac{C_{\mfq}}{\rcn} \sum_{i = 0}^{\lfloor \frac{\log t}{\log 2} \rfloor - i_0} \bigg( \frac{ yz 2^i}{ t \log y} \bigg)
\cdot \frac{t}{2^{i + 1}} \log(t/2^{i + 1})^{-1 + 1/\rcn}.
$$
Here $i_0$ is a constant, depending on $q$, such that
Lemma \ref{lemma_1_mod_m} gives an asymptotic estimate for $x \gg 2^{i_0}$.
We now simplify and approximate the sum by the corresponding integral,
and conclude that
\begin{equation}\label{approx_with_integral}
|S| \gg \frac{C_{\mfq} yz}{\rcn \log y} \int_0^{\frac{\log t}{\log 2} - i_0}
\big( \log t - s \log 2 \big)^{-1 + 1/\rcn} ds
\end{equation}
$$= \frac{C_{\mfq} yz}{(\log 2)(\log y)} \bigg( (\log t)^{1/\rcn} - 
(i_0 \log 2)^{1/\rcn} \bigg)
\gg \frac{C_{\mfq} yz}{\log y} (\log t)^{1/\rcn}.$$

Elements of $T$ come in three types: multiples of $\mfp_0$ and $\mfp_1$, multiples of a prime
of norm greater than $y$, or products of a unit and elements whose norms are less than $t$
and are congruent to 1 modulo $q$. We write $T', T'', T'''$ for these 
subsets of $T$ respectively and we will estimate each in turn.
We have $|T'| \ll yz/\log y$ because $\N \mfp_0, \N \mfp_1 \gg \log y$.
For $T''$, we have that
$$|T''| \leq \sum_{i = 0}^{\lceil \frac{\log (9z)}{\log 2} \rceil - i_0}
\Big(\pi_1(2^{i + 1} y) - \pi_1(2^i y)\Big)
\calS(9z/2^i) + \Big( \pi_1(9yz) - \pi_1(yz/2^{i_0})\Big) \calS(9 \cdot 2^{i_0}).
$$
$$
\ll \sum_{i = 0}^{\lceil \frac{\log (9z)}{\log 2} \rceil - i_0}
\bigg( \frac{ 2^i \omega y}{h_K \log y} \bigg)
\cdot \frac{C_{\mfq} z}{2^i} (\log(9z/2^i))^{-1 + 1/\rcn}
+ O_q\Big( \frac{yz}{\log y} \Big).
$$
In the above, $\pi_1(x)$ counts the number of principal prime ideals of norm $\leq x$.
Estimating in the same way as in (\ref{approx_with_integral}),
we conclude that
$$
|T''| \ll C_{\mfq} \phi(q) \frac{yz (\log z)^{1/\rcn}}{\log y}.
$$
To count elements $T'''$ we apply Lemma \ref{lemma_Psi}.
We choose (as in \cite{shiu})
\begin{equation}\label{choice_t}
t = \exp\Big(\frac{\log y \log \log \log y}{4 \log \log y}\Big),
\end{equation}
and the lemma implies that
$$|T'''| = \omega \Psi(yz, t) \ll
yz (\log t)^2 \exp(-4 \log \log y + o(\log \log y)) \ll \frac{yz}{\log y}.$$
Putting these estimates
together we conclude that
\begin{equation}
|T| \ll C_{\mfq} \phi(q) \frac{yz (\log z)^{1 / \rcn}}{\log y}.
\end{equation}
If $y$ is large in terms of $K$, then the implied constant does not depend on $K$.

Write $P_1$ for the number of primes in $M$ (henceforth ``good primes'') 
congruent to $ua$ modulo $q$
for any unit $u \in \calO_K$, and write $P_2$ for the number of primes 
(``bad primes'') in $M'$ not congruent
to $ua$ for any $u$. By Theorem \ref{pnt_ap_good}, $P_1$ and $P_2$
are determined by $|S|$ and $|T|$, up to an error term which can be made
small by choosing large $x$ and $D$. We therefore conclude that 
\begin{equation}\label{eqn_p1_est}
P_1 \gg C_{\mfq} \frac{yz (\log t)^{1/\rcn}}{\log y} \frac{\N Q^D}{\phi(Q) \log (\N Q^D)}
\end{equation}
and
$$P_2 \ll C_{\mfq} \phi(q) \frac{yz (\log z)^{1/\rcn}}{\log y} \frac{\N Q^D}{\phi(Q) \log (\N Q^D)}.$$
\\
\\
{\bf Finding a bubble of congruent primes.} 
We will split into two cases and compare numbers of good and bad primes. Throughout,
we count all bad primes appearing in $M'$ (which contains $M$), but only those good primes
appearing in $M$.

In the first case the majority of good primes occur in rows
containing at least one bad prime, in which case the proportion of good to bad
primes in some such row of $M'$ is $\gg |S|/|T|$. These primes all occur in some circle
in $\mathbb{C}$ of radius $3 \sqrt{yz}$, and applying Proposition \ref{good_bad_points}
we see that this circle contains a subcircle with $\gg |S|/|T|$ good primes and
no bad primes, which is our desired bubble of congruent primes. The number of
primes in the bubble will be
$$\gg |S|/|T| \gg \frac{1}{\phi(q)}\Big(\frac{\log t}{\log z}\Big)^{1/\rcn}.$$

In the second case, the majority of good primes occur in rows not containing any
bad primes. These such rows then constitute bubbles of congruent primes of
radius $3\sqrt{yz}$, and at least one will contain $\gg P_1 / R$ primes, where $R$ denotes
the number of rows, i.e., the number of elements of $\calO_K$ with norm in
$(\N Q^{D-1}, 2 \N Q^{D - 1})$. As $\calO_K$ forms a lattice in $\mathbb{C}$
we have $R \sim C_K \N Q^{D - 1}$ for some constant
$C_K$ depending on $K$.
Using (\ref{eqn_p1_est}), we see that some row of $M$
 will be a bubble containing
$$\gg_K C_{\mfq} \frac{yz (\log t)^{1/\rcn}}{\log y} \frac{\N Q}{\phi(Q) \log (\N Q^D)}$$
primes.
Now we have
$$\log(\N Q) \ll \sum_{\N \mfp \leq y} \log (\N \mfp) \ll y,$$
and
\begin{equation}\label{ineq_mertens}
\frac{\N Q}{\phi(Q)} = \frac{\N \mfq}{\phi(\mfq)} 
\prod_{\mfp \in \calP} \Big(1 - \frac{1}{\N \mfp}\Big)^{-1}
\gg_{\mfq} \log y (\log t)^{-1/\rcn}.
\end{equation}
To prove (\ref{ineq_mertens}), one can use a result of Rosen
(Theorem 4 of \cite{rosen}, along with 
the result of Landau cited immediately afterwards). The result is then easily
proved, provided that the dependence on $\mfq$ (and $K$) is allowed.

Combining these results, we conclude that this bubble contains $\gg_{\mfq} z$ 
primes. Therefore, our argument produces a bubble
of
$$\gg \min \Big(\frac{1}{\phi(q)}\Big(\frac{\log t}{\log z}\Big)^{1/\rcn}, C'_{\mfq} z\Big)$$
congruent primes, for a constant $C'_{\mfq}$ depending on $\mfq$. Our theorem follows by choosing $z = \log \log (\N Q)$.

\end{document}